\documentclass[11pt]{amsart}
\usepackage[utf8]{inputenc}
\usepackage[T1]{fontenc}
\usepackage{amsmath}
\usepackage{amssymb}
\usepackage{eulerpx,bm}
\usepackage{comment}
\usepackage{longtable}
\usepackage{xcolor,soul} 
\usepackage{colortbl}
\usepackage{hyperref}

\newtheorem*{obs*}{Observation}
\newtheorem{thm}{Theorem}[section]
\newtheorem{prop}[thm]{Proposition}
\newtheorem{lem}[thm]{Lemma}

\newtheorem{innercustomthm}{Theorem}
\newenvironment{customthm}[1]
{\renewcommand\theinnercustomthm{#1}\innercustomthm}
{\endinnercustomthm}

\theoremstyle{definition}
\newtheorem{defn}[thm]{Definition}

\theoremstyle{remark}

\newtheorem{rem}[thm]{Remark}

\numberwithin{equation}{section}

\newcommand{\R}{\mathbb{R}}  
\newcommand{\C}{\mathbb{C}}  
\newcommand{\N}{\mathbb{N}}  
\newcommand{\PP}{\mathbb{P}}  

\begin{document}

\title[\tiny Intermediate Cones between the Cones of PSD Forms and SOS]{\small Intermediate Cones between the Cones of Positive Semidefinite Forms and Sums of Squares}
\author{Charu Goel}
\address{Department of Basic Sciences, Indian Institute of Information Technology, Nagpur, India}
\email{charugoel@iiitn.ac.in}
\author{Sarah Hess}
\address{Department of Mathematics and Statistics, University of Konstanz, Konstanz, Germany}
\email{Sarah.Hess@uni-konstanz.de}
\author{Salma Kuhlmann}
\address{Department of Mathematics and Statistics, University of Konstanz, Konstanz, Germany}
\email{Salma.Kuhlmann@uni-konstanz.de}
\begin{abstract}
The cone $\mathcal{P}_{n+1,2d}$ ($n,d\in\N$) of all positive semidefinite (PSD) real forms in $n+1$ variables of degree $2d$ contains the subcone $\Sigma_{n+1,2d}$ of those that are representable as finite sums of squares (SOS) of real forms of half degree $d$. Hilbert \cite{Hilbert} proved that these cones coincide exactly in the \textit{Hilbert cases} $(n+1,2d)$ with $n+1=2$ or $2d=2$ or $(n+1,2d)=(3,4)$. To establish the strict inclusion $\Sigma_{n+1,2d}\subsetneq\mathcal{P}_{n+1,2d}$ in any non-Hilbert case, one can show that verifying the assertion in the \textit{basic non-Hilbert cases} $(4,4)$ and $(3,6)$ suffices.

In this paper, we construct a filtration of intermediate cones between $\Sigma_{n+1,2d}$ and $\mathcal{P}_{n+1,2d}$. This filtration is induced via the Gram matrix approach \cite{CLR} on a filtration of irreducible projective varieties $V_{k-n}\subsetneq \ldots \subsetneq V_n \subsetneq \ldots \subsetneq V_0$ containing the Veronese variety. Here, $k$ is the dimension of the vector space of real forms in $n+1$ variables of degree $d$.
By showing that $V_0,\ldots,V_n$ are varieties of minimal degree, we demonstrate that the corresponding intermediate cones coincide with $\Sigma_{n+1,2d}$. Likewise, for the special case when $n=2$, $V_{n+1}$ is also a variety of minimal degree and the corresponding intermediate cone also coincides with $\Sigma_{n+1,2d}$. We moreover prove that, in the non-Hilbert cases of $(n+1)$-ary quartics for $n\geq 3$ and $(n+1)$-ary sextics for $n\geq 2$, all the remaining cone inclusions are strict. 
\end{abstract}
\maketitle
\vspace*{-0.6cm}\section{Introduction}
\subsection{Overview}
A prominent question in real algebraic geometry is whether a given positive semidefinite (PSD) polynomial of even degree can be represented by a finite sum of squares (SOS) of half degree polynomials. It suffices to consider this question for forms, since the PSD and SOS properties are preserved under homogenization \cite[Proposition 1.2.4.]{Marshall}.
In 1888, Hilbert \cite{Hilbert} classified all cases $(n+1,2d)$ in which $\mathcal{P}_{n+1,2d}=\Sigma_{n+1,2d}$, namely the \textit{Hilbert cases} $(2,2d)$, $(n+1,2)$ and $(3,4)$. For the remaining cases, he proved the existence of PSD-not-SOS forms by reducing this investigation to the \textit{basic non-Hilbert cases} $(4,4)$ and $(3,6)$ (using an argument that allows him to increase the number of variables and the even degree while preserving the PSD-not-SOS property). In the cases $(4,4)$ and $(3,6)$, he gave a non-constructive proof for the existence of PSD-not-SOS forms. The first explicit PSD-not-SOS examples of a ternary sextic and a quaternary quartic were given by Motzkin \cite{Motzkin} and Robinson \cite{Rob} respectively. More examples of PSD-not-SOS ternary sextics and quaternary quartics were given by Choi and Lam \cite{CL, CL2}. 
	
In this paper, we use the Gram matrix method \cite{CLR} to construct a filtration of intermediate cones between $\Sigma_{n+1,2d}$ and $\mathcal{P}_{n+1,2d}$ along irreducible projective varieties containing the Veronese variety. For determining the intermediate cones coinciding with $\Sigma_{n+1,2d}$, we use a result on varieties of minimal degree by Blekherman et al.\ \cite{Blek}. On the other hand, we identify the strictly separating intermediate cones in the non-Hilbert cases $(n+1,4)_{n\geq 3}$ and $(n+1,6)_{n\geq 2}$. To this end, we reconsider the Motzkin and the Choi-Lam forms in light of the Gram matrix method. In the next paper \cite{Sarah}, we shall extend the above results by identifying all strictly separating intermediate cones in any non-Hilbert case via a \textit{degree-jumping principle}. Both papers together will provide a complete generalization of Hilbert's 1888 Theorem along a filtration of irreducible projective varieties containing the Veronese variety.

\subsection{Preliminaries}
For $n,\, d\in\N$, let $\mathcal{F}_{n,d}$ denote the vector space of real forms in $n$ variables of degree $d$; $p\in\mathcal{F}_{n,d}$ is called a \textit{$n$-ary $d$-ic}.
A form $f\in\mathcal{F}_{n+1,2d}$ is \textit{locally positive semidefinite on $K\subseteq\R^{n+1}$} if $f(x)\geq 0$ holds for all $x\in K$. 
In particular, if $K=\R^{n+1}$, then $f$ is \textit{(globally) positive semidefinite (PSD)}.
Moreover, a form $f\in\mathcal{F}_{n+1,2d}$ is a \textit{sum of squares (SOS)}, if there exist some $g_1,\ldots,g_s\in\mathcal{F}_{n+1,d}$ ($s\in\N$) s.t.\ $f=\sum\limits_{i=1}^sg_i^2$.
The cone of all PSD resp.\ SOS forms in $\mathcal{F}_{n+1,2d}$ is denoted by $\bm{\mathcal{P}_{n+1,2d}}$ resp.\ $\bm{\Sigma_{n+1,2d}}$. 

\subsubsection*{\textbf{The Gram Matrix Method and the Veronese Embedding}}
$\Sigma_{n+1,2d}$ can be studied by using the Gram matrix method introduced in \cite{CLR}. The crux of this method is that a SOS representation $f=\sum\limits_{i=1}^sg_i^2$ ($s\in\N$) of a form $f$ corresponds to a real, symmetric, positive semidefinite matrix whose entries come from the coefficients of the half degree forms $g_1,\ldots,g_s$. 

Set
\begin{eqnarray*} k \colon \N \times \N &\to & \N \\ (n,d) &\mapsto&k(n,d):=\dim(\mathcal{F}_{n+1,d})-1=\binom{n+d}{n}-1.
 \end{eqnarray*} We write $k$ instead of $k(n,d)$ whenever $(n,d)$ is clear from the context.
Consider the linear bijective map 
\begin{eqnarray} \label{Q} Q\colon\mathrm{Sym}_{k+1}(\R) &\to&\mathcal{F}_{k+1,2} \\
	\nonumber			A 							&\mapsto&q_A(Z):=ZAZ^t,
\end{eqnarray} where $\mathrm{Sym}_{k+1}(\R)$ is the vector space of all $(k+1)\times(k+1)$ symmetric matrices with real entries and $Z:=(Z_0,\ldots,Z_k)$. We order the set of exponents of monomials in $\mathcal{F}_{n+1,d}$, that is $\{\alpha\in\N_0^{n+1}\mid\vert\alpha\vert=d\}$ (where $\N_0:=\N\cup\{0\}$), lexicographically and recognize that it naturally consists of exactly $k+1$ elements. Thus, we obtain the ordered set of exponents $\{\alpha_0,\ldots,\alpha_k\}$ and write $\alpha_i=(\alpha_{i,0},\ldots,\alpha_{i,n})$. 
Lastly, we enumerate the monomial basis of $\mathcal{F}_{n+1,d}$ along the $\alpha_i$ by setting $m_i(X):=X^{\alpha_i}$ for $i=0,\ldots,k$. Let us now consider the linear surjective map
\begin{eqnarray*} \mathcal{G} \colon \mathrm{Sym}_{k+1}(\R) & \to & \mathcal{F}_{n+1,2d} \\
									A				& \mapsto & f_A(X):=q_A(m_0(X)\ldots m_k(X)).
\end{eqnarray*} We call $\mathcal{G}$ the \textit{Gram map}. For any $f\in\mathcal{F}_{n+1,2d}$, the preimage $\mathcal{G}^{-1}(f)$ is always nonempty. Any matrix $A\in\mathcal{G}^{-1}(f)$ is called a \textit{Gram matrix associated to $f$} and $q_A$ its corresponding \textit{quadratic form (associated to $f$)}. In particular, we can fix a generic $A_f\in\mathcal{G}^{-1}(f)$.
Clearly, any form $f\in\mathcal{F}_{n+1,2d}$ can be interpreted as a form $g\in\mathcal{F}_{m+1,2d}$ for $m\geq n$. The same is true for the associated Gram matrices as the following observation clarifies: 

\begin{obs*} \label{obs:genf} Let $n,m,d\in\N$, $n\leq m$, $i_0,\ldots,i_n\in\{0,\ldots,m\}$ such that $i_0<\ldots<i_n$. For $f\in\mathcal{F}_{n+1,2d}$, define $g(X_0,\ldots,X_m):=f(X_{i_0},\ldots,X_{i_n})$ in $\mathcal{F}_{m+1,2d}$ and set $I:=\{i\in\{0,\ldots,k(m,d)\}\mid \alpha_{i,s}=0 \mbox{ for } s\neq i_0,\ldots,i_n\}$.
	
	For any $B:=(b_{i,j})_{0\leq i,j\leq k(m,d)}\in\mathcal{G}^{-1}(g)$, set $A:=B_I\in\mathrm{Sym}_{k(n,d)+1}(\R)$ to be the submatrix of $B$ given by $B_I:=(b_{i,j})_{i,j\in I}$, then $A$ is a Gram matrix associated to $f$. 
	
	Conversely, for any $A\in\mathcal{G}^{-1}(f)$, expand $A$ to a matrix $B\in\mathrm{Sym}_{k(m,d)+1}(\R)$ with $B_I:=(b_{i,j})_{i,j\in I}=A$ and zero elsewhere, then $B$ is a Gram matrix associated to $g$.
\end{obs*}

The map $\mathcal{G}$ is not injective in general and hence $\mathcal{G}^{-1}(f)$ is not a singleton. In fact, $\mathcal{G}^{-1}(f)$ is the coset of $A_f$ w.r.t.\ $\ker(\mathcal{G})$. Let us characterize the kernel of $\mathcal{G}$ via the
\textit{(projective) Veronese embedding}
\begin{eqnarray*} V\colon \PP^n & \to & \PP^k \\
							\, [x] & \mapsto & [m_0(x):\ldots:m_k(x)].
\end{eqnarray*}
Observe that $V$ is not surjective in general. Set
\begin{equation} \label{S} \mathcal{S}_V:=\{Z_iZ_j-Z_sZ_t \mid i,j,s,t\in\{0,\ldots,k\}\ \colon \alpha_i+\alpha_j=\alpha_s+\alpha_t\},
\end{equation} then $V(\PP^n)=\mathcal{V}(\mathcal{S}_V)$ \cite[Lemma 2.35]{G}, \cite[Satz 4.3.13]{Plau}. We call $V(\PP^n)$ the \textit{Veronese variety}. 

We now build a bridge between the real affine and the complex projective formulations. For any projective set $W$, let $W(\R)$ denote the \textit{set of real points of $W$} and assume $x$ to be a real representative for $[x]\in W(\R)$. Observe that $V(\PP^n(\R))=V(\PP^n)(\R)$ and
	\begin{eqnarray*}
		\ker(\mathcal{G})&=&\{A\in\mathrm{Sym}_{k+1}(\R) \mid q_A \mbox{\ vanishes\ on\ } V(\PP^n)(\R)\}.
	\end{eqnarray*} Consequently,  for any $f\in\mathcal{F}_{n+1,2d}$,
\begin{eqnarray} \label{Kern}
	\mathcal{G}^{-1}(f)&=& \{A\in\mathrm{Sym}_{k+1}(\R)\mid q_{A-A_f} \mbox{\ vanishes\ on\ } V(\PP^n)(\R)\}.
\end{eqnarray}

\subsubsection*{\textbf{Cone Filtration between $\Sigma_{n+1,2d}$ and $\mathcal{P}_{n+1,2d}$}} \label{SecConeFilt}
We now extend the (locally) positive semidefinite notion to a projective setting. For any projective variety $W\subseteq\PP^k$, we call a form $q\in\mathcal{F}_{k+1,2}$ \textit{locally positive semidefinite (locally PSD) on $W(\R)\subseteq\PP^k(\R)$} if $q(x)\geq 0$ holds for any $[x]\in W(\R)$. In particular, if $W(\R)=\PP^k(\R)$, then we refer to $q$ as \textit{positive semidefinite (PSD)} or \textit{nonnegative}.

A form $f\in\mathcal{F}_{n+1,2d}$ is SOS if and only if there exists a Gram matrix $A$ associated to $f$ such that $q_A$ is PSD \cite[Theorem 2.4]{CLR}. Hence,
\begin{eqnarray}
	\label{CharSOS2}
	\Sigma_{n+1,2d}
	&=&\{f\in\mathcal{F}_{n+1,2d}\mid \exists\ A\in \mathcal{G}^{-1}(f) \colon q_A\vert_{_{\PP^{k}(\R)}}\geq 0\}.
\end{eqnarray} 
Likewise, a form $f\in\mathcal{F}_{n+1,2d}$ is $PSD$ if and only if for any Gram matrix associated to $f$ the corresponding quadratic form is locally PSD on $V(\PP^n)(\R)$. Equivalently, there exists a Gram matrix $A$ associated to $f$ such that $q_A$ is locally PSD on $V(\PP^n)(\R)$. Indeed, assume that for some $A\in\mathcal{G}^{-1}(f)$, the corresponding quadratic form $q_A$ is locally PSD on $V(\PP^n)(\R)$, then $f(X)=q_A(m_0(X),\ldots,m_k(X))=(m_0(X),\ldots,m_k(X))A(m_0,\ldots,m_k(X))^t$ verifies that $f$ is PSD. Equivalently, any Gram matrix associated to $f$ is locally PSD on $V(\PP^n)(\R)$. Thus,
\begin{eqnarray} \label{CharPSD} 
	\mathcal{P}_{n+1,2d} &=&\{f\in\mathcal{F}_{n+1,2d}\mid \exists\ A\in \mathcal{G}^{-1}(f) \colon q_A\vert_{_{\mathcal{V}(\PP^n)(\R)}}\geq 0\}.
\end{eqnarray}
In (\ref{CharSOS2}), the SOS property
of a form in $\mathcal{F}_{n+1,2d}$ was characterized by a PSD
restriction for corresponding quadratic forms on the set of real points over the variety $\PP^k$. Analogously, in (\ref{CharPSD}), the PSD property was described by a local PSD restriction for corresponding quadratic forms on the set of real points of the variety $V(\PP^n)$. For any $i\in\{0,\ldots,k-n\}$, set $$H_i:=\{[z]\in\PP^k\mid\exists\ x\in \C^{n+1}\colon(z_0,\ldots,z_{n+i})=(m_0(x),\ldots,m_{n+i}(x))\}$$ and let $V_i\subseteq \PP^k$ be the Zariski closure of $H_i$ in $\PP^k$ (that is, the smallest projective variety in $\PP^k$ containing $H_i$ ). In particular, $V_i=\mathcal{V}(\mathcal{I}(H_i))$, where $\mathcal{I}(H_i)$ denotes the vanishing ideal of $H_i$. Clearly, $V_{k-n}=V(\PP^n)$ and $V_0=\PP^k$, so that we obtain a filtration of projective varieties $$V(\PP^n)=V_{k-n} \subsetneq \ldots \subsetneq V_0=\PP^k$$ in which each inclusion is strict. Moreover, each inclusion in the corresponding filtration of sets of real points is also strict. That is
\begin{eqnarray} \label{SetOfRealPointsFiltration} V(\PP^n)(\R)=V_{k-n}(\R) \subsetneq \ldots \subsetneq V_0(\R)=\PP^k(\R).
\end{eqnarray}

This allows us to induce a filtration of cones between $\Sigma_{n+1,2d}$ and $\mathcal{P}_{n+1,2d}$ by following the Gram matrix method. In particular, for $i=0,\ldots,k-n$, setting
$$C_i:=\{f\in\mathcal{F}_{n+1,2d}\mid \exists\ A\in \mathcal{G}^{-1}(f) \colon q_A\vert_{_{V_i(\R)}}\geq 0\},$$
we obtain
\begin{equation} \label{ConeFiltration}
	\Sigma_{n+1,2d}=C_0\subseteq \ldots \subseteq C_{k-n}=\mathcal{P}_{n+1,2d}.
\end{equation} Even though the inclusions between sets of real points in (\ref{SetOfRealPointsFiltration}) are all strict, the induced cone inclusions in (\ref{ConeFiltration})need not be.

If $(n+1,2d)$ is a Hilbert case, then (\ref{ConeFiltration}) collapses to one and the same cone, that is,
$$\Sigma_{n+1,2d}=C_0=C_1=\ldots=C_{k-n}=\mathcal{P}_{n+1,2d}.$$ On the other hand, in any non-Hilbert case $(n+1,2d)$
at least one inclusion in (\ref{ConeFiltration}) is strict since $\Sigma_{n+1,2d}$ is a proper subcone of $\mathcal{P}_{n+1,2d}$. However, it is not clear how many among these inclusions are strict and which ones.
Our main goal is to identify all strict inclusions in (\ref{ConeFiltration}). 

\subsection{Structure of the Paper and Main Results} \label{Structure}
This paper is structured as follows. In Section \ref{minnum}, we give an explicit description of $V_0,\ldots,V_{k-n}$ in Theorem \ref{thm:explicitdescription} along the quadratic generators of the Veronese variety in $\mathcal{S}_V$.
In Section \ref{collapse}, we investigate potential equalities in the cone filtration $\Sigma_{n+1,2d}=C_0\subseteq \ldots \subseteq C_n$ for $n,d\in\N$. In the special case of $n=2$, we extend these considerations a step further to $C_n\subseteq C_{n+1}$. This leads us to our first main result:
\begin{customthm}{A} \label{thm:collapse} If $n,\, d\in\N$ and $i=0,\ldots,n-1$, then $C_i=C_{i+1}$. Moreover, if $n\leq 2$, then also $C_n=C_{n+1}$.
\end{customthm}

\noindent For proving this result, our main tool is Proposition \ref{lem:degreejump}
that establishes the applicability of a result on varieties of minimal degree from \cite{Blek}.
In Section \ref{quaternaryquartics}, we consider quaternary quartics and ternary sextics. We construct explicit examples of separating forms in these cases (see Theorem \ref{thm:sepquartics} and Theorem \ref{thm:sepsextics}) thereby establishing all the remaining inclusions in (\ref{ConeFiltration}) being strict. In Section \ref{quartics}, we generalize these findings to $(n+1)$-ary quartics for $n\geq 4$ and $(n+1)$-ary sextics for $n\geq 3$ by proving our second main result (see Theorem \ref{thm:degreeraise} below), which allows us to maintain proper cone inclusions when going over to higher degrees for a fixed number of variables.

\begin{customthm}{B} \label{thm:degreeraise} If $n,d\in\N$ and $i=0,\ldots,k(n,d)-1$ such that $C_i\subsetneq C_{i+1}$ as subcones of $\mathcal{P}_{n+1,2d}$, then $C_i\subsetneq C_{i+1}$ also as subcones of $\mathcal{P}_{n+1,2\delta}$ for $\delta\geq d$.
\end{customthm}

\noindent Using this result and our observations made in Section \ref{quaternaryquartics}, we deduce in Theorem \ref{thm:conesdiffer} and Theorem \ref{thm:conesdifferSextics} that all the remaining cone inclusions in (\ref{ConeFiltration}) are also strict:

\begin{customthm}{C} \label{thm:conesdiffer} For $(n+1,4)_{n\geq 4}$ and $i=n,\ldots,k-n-1$, the strict inclusion $C_i\subsetneq C_{i+1}$ holds.
\end{customthm}

\begin{customthm}{D} \label{thm:conesdifferSextics} For $(n+1,6)_{n\geq 3}$ and $i=n,\ldots,k-n-1$, the strict inclusion $C_{i} \subsetneq C_{i+1}$ holds.
\end{customthm}	
\section{Explicit Description of the Filtration of Varieties containing the Veronese Variety} \label{minnum}
In a first step towards a complete determination of all strict inclusions in the induced cone filtration (\ref{ConeFiltration}), we give an explicit description of the projective varieties $V_0,\ldots,V_{k-n}$ along the quadratic forms in $\mathcal{S}_V$.

Let $(n+1,2d)$ be a non-Hilbert case and for $i=0,\ldots,k-n$, set $$s_i:=\min\{s\in\{1,\ldots,k\}\mid \exists\ t\in\{s,\ldots,k\}\colon \alpha_{s}+\alpha_{t}=\alpha_0+\alpha_{n+i}\}.$$ Fix $t_i\in\{s_i,\ldots,k\}$ such that $\alpha_{s_i}+\alpha_{t_i}=\alpha_0+\alpha_{n+i}$ correspondingly and consider the quadratic form $q_i(Z):=Z_0Z_{n+i}-Z_{s_i}Z_{t_i}\in\mathcal{S}_V$. Define the sets $K_0:=\mathcal{V}(\emptyset)=\C^k$ and $K_i:=\mathcal{V}(q_1(1,Z_1,\ldots,Z_k),\ldots,q_i(1,Z_1,\ldots,Z_k))\subseteq \C^k$ for $i=1,\ldots,k-n$. Moreover, for $i=0,\ldots,k-n$, set $W_i$ to be the Zariski closure of $K_i$ under the embedding 
$$\begin{array}{cccc} \phi\colon&\C^k&\to&\PP^k \\ &(z_1,\ldots,z_k)&\mapsto& [1:z_1:\ldots:z_k].
\end{array}$$

\begin{lem} \label{lem:Kiirred} For $i=0,\ldots,k-n$, the projective variery $W_i$ is irreducible and has codimension $i$.
\end{lem}
\begin{proof} The affine variety $K_i$ is isomorphic to the ireducible affine variety $\C^{k-i}$. Thus, the projective closure $W_i$ of $K_i$ under the embedding $\phi$ is also irreducible and has codimension $i$.
\end{proof}
		
\begin{thm} \label{thm:explicitdescription} For $i=0,\ldots,k-n$, then $W_i=V_i$.
\end{thm}
	\begin{proof} Clearly, $W_0=\PP^k=V_0$ and $W_{k-n}$ is a subvariety of the Veronese variety. Furthermore, by Lemma \ref{lem:Kiirred}, $W_{k-n}$ is irreducible and has dimension $n$. The same is true for $V(\PP^n)$ \cite[Example 13.4]{JHarris}. Hence, it follows $W_{k-n}=V(\PP^n)=V_{k-n}$.
		
	Let $i=1,\ldots,k-n-1$ and observe that $V_i=\overline{\mathcal{Z}}\cup W_i$, where $\overline{\mathcal{Z}}$ is the Zariski closure of $\mathcal{Z}:=\{[z]\in H_i \mid z_0=0\}$. Hence, it suffices to show $\overline{\mathcal{Z}}\subseteq W_i$. 
	
	According to Buchberger's algorithm for a Gröbner basis construction \cite[Chapter 2 §7 Theorem 2]{CLO}, we fix a set of forms $F\subseteq\C[Z_0,\ldots,Z_{n+i}]$ such that $W_i=\mathcal{V}(F)$. Thus, using the projective Hilbert-Nullstellensatz \cite[Chapter 8 §3 Theorem 9]{CLO}, 
	it is enough to show $\sqrt{\langle F \rangle}\subseteq\sqrt{\mathcal{I}(\mathcal{Z})}$ for which it suffices to verify
	$\langle F \rangle \subseteq\mathcal{I}(\mathcal{Z})$. 
	
	Let $f\in F$ and $[z]\in \mathcal{Z}$ be arbitrary but fixed. Hence, for some $x\in\C^{n+1}$ with $x_0=0$, we get $(z_0,\ldots,z_{n+i})=(m_0(x),\ldots,m_{n+i}(x))$. Observe that $f(z)=f(\tilde{z})$ for $\tilde{z}:=(m_0(x),\ldots,m_k(x))$. Recall $W_i\supseteq V(\PP^n)$, which implies $\mathcal{I}(W_i)\subseteq\mathcal{I}(V(\PP^n))=\sqrt{\langle \mathcal{S}_V \rangle}$. Hence, we can fix some $m\in\N$ and $g\in\langle \mathcal{S}_V \rangle$ such that $f^m=g$. It follows $f(z)^m=f(\tilde{z})^m=(f^m)(\tilde{z})=g(\tilde{z})=0$. So, $f\in\mathcal{I}(\mathcal{Z})$. 
	\end{proof}

	\begin{rem} Since $\dim(\PP^k)-\dim(V(\PP^n))=k-n$, the filtration of irreducible projective varieties $V(\PP^n)=V_0\subsetneq \ldots \subsetneq V_{k-n}=\PP^k$ is of maximal length w.r.t.\ $\subsetneq$.
	\end{rem}
\section{Equality of the first $n+1$ (respectively $n+2$ if $n\leq 2$) Cones} \label{collapse}
Our investigation of the cone inclusions $\Sigma_{n+1,2d}=C_0\subseteq\ldots\subseteq C_n$ (respectively $\Sigma_{n+1,2d}=C_0\subseteq\ldots\subseteq C_n\subseteq C_{n+1}$ if $n=2$) makes use of the result \cite[Theorem 1.1]{Blek} on varieties of minimal degree. Recall that for any irreducible nondegenerate projective variety $W$, $\deg(W)\geq1+\mathrm{codim}(W)$ \cite[Proposition 0]{EH}.
\begin{defn} Let $W$ be an irreducible nondegenerate projective variety with $\deg(W)=1+\mathrm{codim}(W)$, then $W$ is called a projective variety of \textit{minimal degree}.
\end{defn}

Theorem \ref{thm:collapse} will then follow from Proposition \ref{thm:suff1} and Proposition \ref{lem:degreejump} below.

\begin{prop} \label{thm:suff1} Let $n,\, d\in\N$. If $W\subseteq\PP^k$ is a totally real irreducible nondegenerate projective variety of minimal degree with $V(\PP^n)\subseteq W$, then $$\{f\in\mathcal{F}_{n+1,2d}\mid \exists\ A\in \mathcal{G}^{-1}(f) \colon q_A\vert_{_{W(\R)}}\geq 0\}=\Sigma_{n+1,2d}.$$
\end{prop}
\begin{proof}
	Let $f\in\mathcal{F}_{n+1,2d}$ with $A\in \mathcal{G}^{-1}(f)$ such that $q_A$ is locally PSD on $W(\R)$ be arbitrary but fixed. Thus, by \cite[Theorem 1.1]{Blek}, $q_A\in\R[Z]$ is SOS in $\R[Z]/\mathcal{I}(W)$. Therefore, there exist some linear forms $h_1,\ldots,h_s\in\R[Z]$ ($s\in\N$) with $q_A-\sum\limits_{i=1}^{s} h_i^2 \in \mathcal{I}(W)\subseteq \mathcal{I}(V(\PP^n))$. Set $q(Z):=\sum\limits_{i=1}^{s} h_i^2 \in \Sigma_{k+1,2}$ and fix $B\in Q^{-1}(q)$, then $q_A \equiv q_B$ on $V(\PP^n)(\R)$ testifies $B\in\mathcal{G}^{-1}(f)$. We conclude $f\in \Sigma_{n+1,2d}$.
\end{proof}

To get equality of each inclusion in $\Sigma_{n+1,2d}=C_0\subseteq \ldots \subseteq C_n$ (respectively $\Sigma_{n+1,2d}=C_0\subseteq\ldots\subseteq C_{n+1}$ if $n=2$), it thus remains to verify that $V_n$ (respectively $V_{n+1}$ if $n=2$) are varieties of minimal degree.

\begin{prop} \label{lem:degreejump}  For $n\geq 2$, the following are true:
	\begin{enumerate} 
		\item $\deg(V_i)=i+1$ for $i=0,\ldots,n$
		\item $\deg(V_{n+1})=n+2$ for $n=2$
	\end{enumerate} 
\end{prop}
\begin{proof} \begin{enumerate}
		\item Clearly, $V_0=\PP^k$ has degree $1$ and $V_1=W_1=\mathcal{V}(q_1)$ has degree $2$ since $q_1$ is a quadratic form. It therefore remains to determine $\deg(V_i)$ for $i\in\{2,\ldots,n\}$. 
		
		If $n=2$, then $\deg(V_2)=\deg(W_2)=3$ by a straight forward computation and we are done. 
		
		Thus, w.l.o.g.\ assume $n\geq 3$ and argue inductively. By explicit calculations $\deg(V_2)=\deg(W_2)=3$ and $\deg(V_3)=\deg(W_3)=4$. Now, assume that the assertion already holds up to some $i\in\{3,\ldots,n-1\}$. In the inductive step, we now consider $i+1$.
		
		Recall $\deg(V_{i+1})=\deg(W_{i+1})\geq i+2$ by \cite[Proposition 0]{EH}. Set $$B^{[2]}:=\{Z_0,Z_1,Z_{n+1},Z_{n+s}-Z_sZ_{n+2}\mid s=3,\ldots,i\}$$ in $\R[Z_0,Z_1,Z_3,\ldots,Z_k]$ and let $W^{[2]}$ be the projective closure of the affine variety $H^{[2]}:=\mathcal{V}\left(B^{[2]}\right)\subseteq\C^{k}$ under the embedding 
	$$\begin{array}{cccc}
		\phi^{[2]} \colon & \C^k &\to& \PP^k \\ &(z_0,z_1,z_3,\ldots,z_n)&\mapsto&(z_0,z_1,1,z_3,\ldots,z_n).
	\end{array} $$
	Observe that $W_{i+1}, \mathcal{V}(Z_0,\ldots,Z_{i})$ and $W^{[2]}$ are three pairwise distinct irreducible components of $W_{i}\cap\mathcal{V}(q_{i+1})$. Thus, Bézout's Theorem \cite[Theorem 7.7]{Hart} together with the inductive assumption yields that for some $c_1,c_2,c_3\in\N$
	\begin{eqnarray*} \hspace*{1.4cm} 2(i+1)&=&\deg(\mathcal{V}(q_{i+1}))\deg(W_{i})\\
		&\geq& c_1\deg(W_{i+1})+c_2\deg(\mathcal{V}(Z_0,\ldots,Z_{i}))+c_3\deg\left(W^{[2]}\right) \\
		&\geq&\deg(W_{i+1})+\deg(\mathcal{V}(Z_0,\ldots,Z_{i}))+\deg\left(W^{[2]}\right) \\
		&=&\deg(W_{i+1})+1+(i-1) \\
		& \geq & (i+2)+1+(i-1)=2(i+1).
	\end{eqnarray*} So, $\deg(W_{i+1})+1+(i-1)=2(i+1)$ and the irreducible decomposition of $V_i\cap\mathcal{V}(q_{i+1})$ is $V_{i+1}\cup\mathcal{V}(Z_0,\ldots,Z_{i})\cup W^{[2]}$. Therefore, we conclude $\deg(V_{i+1})=\deg(W_{i+1})=i+2$.
	\item Set $W:=\mathcal{V}(Z_3Z_0-Z_1^2,Z_0Z_4-Z_1Z_2,Z_0Z_5-Z_2^2,Z_3Z_5-Z_4^2)\subseteq\PP^k$ and observe that $W=W_3$. Thus, we have $\dim(W)=k-3$ and therefore by \cite[Proposition 0]{EH} $\deg(W)\geq 1+\mathrm{codim}(W)=4$. 
	Moreover, observe that $W_3$ and $\mathcal{V}(Z_0,Z_1,Z_2)$ are two distinct irreducible components of $W_{2}\cap\mathcal{V}(Z_0Z_5-Z_2^2)$.
	Hence, we observe by Bézout's Theorem \cite[Theorem 7.7]{Hart} that
	\begin{eqnarray*}
		6&=&3\cdot2=\deg(W_{2})\cdot\deg(\mathcal{V}(Z_0Z_5-Z_2^2))\\
		&\geq&c_1\deg(W_3)+c_2\deg(\mathcal{V}(Z_0,Z_1,Z_2))\geq\deg(W_3)+1
	\end{eqnarray*} for some $c_1,c_2\in\N$.
	We conclude $\deg(W)=\deg(W_3)\in\{4,5\}$. However, the degree of $W$ is also divisible by $4$. Therefore it follows $\deg(V_3)=\deg(W_3)=\deg(W)=4$.
\end{enumerate}
\end{proof}
\section{Quaternary Quartics and Ternary Sextics} \label{quaternaryquartics}
Applying Theorem \ref{thm:collapse} to $(n+1,2d)=(4,4)$ and $(3,6)$ respectively, the cone filtration (\ref{ConeFiltration}) becomes
\begin{equation} \label{FiltrationQuaternaryQuartics}
	\Sigma_{4,4}=C_0=C_1=C_2=C_3\subseteq C_4 \subseteq C_5 \subseteq C_6=\mathcal{P}_{4,4} \mbox{, and}
\end{equation}
\begin{equation} \label{FiltrationTernarySextics} 
	\Sigma_{3,6}=C_0=C_1=C_2=C_3\subseteq C_4 \subseteq C_5 \subseteq C_6\subseteq C_7=\mathcal{P}_{3,6}
\end{equation}
respectively. We now investigate the remaining inclusions.
\begin{thm} \label{thm:sepquartics}
	If $(n+1,2d)=(4,4)$, then $C_3\subsetneq C_4 \subsetneq C_5 \subsetneq C_6$.
\end{thm}
	\begin{proof} Consider the Choi-Lam quaternary quartic
		$$\mathfrak{C}(X_0,X_1,X_2,X_3):=X_0^2X_1^2+X_0^2X_2^2+X_1^2X_2^2+X_3^4-4X_0X_1X_2X_3.$$ From \cite{CL, CL2}, we know that $\mathfrak{C}\in\mathcal{P}_{4,4}\backslash\Sigma_{4,4}$. Thus, the same holds for the permuted forms $\mathfrak{C}^\sigma:=\mathfrak{C}(X_0,X_3,X_1,X_2)$ and $\mathfrak{C}^\tau:=\mathfrak{C}(X_3,X_1,X_2,X_0)$.
		We now show the following:
		\begin{enumerate} 
			\item $\mathfrak{C}^\sigma\in C_4\backslash C_3$,
			\item $\mathfrak{C}^\tau\in C_5\backslash C_4$,
			\item $\mathfrak{C}\in C_6\backslash C_5$.
		\end{enumerate}
		The above three statements can be shown by similar arguments. Thus, it is enough to prove \textit{(1)}. Observe that
		$$\mathfrak{C}^\sigma=m_1(X)^2+m_3(X)^2+m_6(X)^2+m_7(X)^2-4m_2(X)m_6(X).$$
		\begin{enumerate}
			\item[(a)] To prove that $\mathfrak{C}^\sigma\in C_4$, it suffices to show that there exists a Gram matrix $A:=(a_{i,j})_{0\leq i,j\leq 9}$ associated to $\mathfrak{C}^\sigma$ such that $q_A$ is locally PSD on $\mathcal{V}(q_1,q_2,q_3,q_4)(\R)$. Consider the Gram matrix $A:=(a_{i,j})_{0\leq i,j\leq 9}$ associated to $\mathfrak{C}^\sigma$ given by $$a_{i,j}:=\begin{cases} 1, & \mbox{if } i=j\in\{1,3,6,7\} \\ -2, & \mbox{if } \{i,j\}=\{2,6\} \\ 0, & \mbox{else.}\end{cases}$$
		Observe that for any $[z]\in\mathcal{V}(q_1,q_2,q_3,q_4)(\R)$, we have
		$$q_A(z)=\begin{cases} z_1^2+z_3^2+z_6^2+z_7^2, & \mbox{if } z_0=0 \\ \frac{1}{z_0^2}\mathcal{C}^\sigma(1,z_1,z_2,z_3), & \mbox{else.} \end{cases}$$ This implies $q_A(z)\geq 0$ for any $[z]\in\mathcal{V}(q_1,q_2,q_3,q_4)(\R)$.
		\item[(b)] To prove that $\mathfrak{C}^\sigma\not\in C_3$, it suffices to show that for any Gram matrix $B$ associated to $\mathfrak{C}^\sigma$, $q_B$ is not locally PSD on $\phi(K_3)(\R)$. On contrary, assume that there exists a Gram matrix $B=(b_{i,j})_{0\leq i,j\leq 9}$ associated to $\mathfrak{C}^\sigma$ such that $q_B(z)\geq 0$ for any $[z]\in\phi(K_3)(\R)$. A comparision of coefficients between $\mathfrak{C}^\sigma$ and $q_B(m_0(X),\ldots,m_9(X))$ as well as serveral point evaluations of $q_B$ on $\phi(K_3)(\R)$ now reveals that necessarily $b_{7,7}=1$, $b_{i,j}=b_{j,i}=0$ for $j\in\{7,8,9\},i\in\{0,\ldots,j\}$, $(i,j)\neq (7,7)$. This gives $q_B(z)<0$ for $[z]:=\left[1:\ldots:1:0:0:0\right]\in\phi(K_3)(\R)$, which is a contradiction.
		\end{enumerate}
	\vspace*{-0.4cm}
	\end{proof}
\begin{thm} \label{thm:sepsextics}
	If $(n+1,2d)=(3,6)$, then $C_3\subsetneq C_4\subsetneq C_5\subsetneq C_6 \subsetneq C_{7}$.
\end{thm}
\begin{proof} Consider the Motzkin form
	$$\mathfrak{M}(X_0,X_1,X_2):=X_0^4X_1^2+X_0^2X_1^4+X_2^6-3X_0^2X_1^2X_2^2$$
	and the Choi-Lam ternary sextic
	\begin{eqnarray*}
			\mathfrak{L}(X_0,X_1,X_2)&:=&X_0^4X_1^2 + X_0^2X_2^4 + X_1^4X_2^2 - 3X_0^2X_1^2X_2^2.
		\end{eqnarray*} From \cite{Motzkin} and \cite{CL, CL2}, we know $\mathfrak{M},\mathfrak{L}\in\mathcal{P}_{3,6}\backslash\Sigma_{3,6}$. Thus, the same holds for the permuted forms $\mathfrak{M}^\sigma:=\mathfrak{M}(X_0,X_2,X_1)$ and $\mathfrak{L}^\tau:=\mathfrak{L}(X_0,X_2,X_1)$. We now show the following:
	\begin{enumerate} 
		\item $\mathfrak{M}^\sigma\in C_4\backslash C_3$,
		\item $\mathfrak{L}\in C_5\backslash C_4$,
		\item $\mathfrak{L}^\tau\in C_6\backslash C_5$,
		\item $\mathfrak{M}\in C_7\backslash C_6$.
	\end{enumerate}
	These statements follow analogously as in the proof of Theorem \ref{thm:sepquartics} except that an appropriate choice of the Gram matrix $A:=(a_{i,j})_{0\leq i,j\leq 9}$ associated to $\mathfrak{M}^\sigma$ such that $q_A$ is locally PSD on $\mathcal{V}(q_1,q_2,q_3,q_4)(\R)$ is given by $$a_{i,j}:=\begin{cases}1, & \mbox{if } i=j\in\{2,4,5,6\} \\ 4, & \mbox{if } i=j=3 \\ -2, & \mbox{if } \{i,j\}=\{1,6\} \mbox{ or } \{i,j\}=\{3,5\} \\ 0, & \mbox{else.} \end{cases}$$
\end{proof}

From the cone filtrations (\ref{FiltrationQuaternaryQuartics}), (\ref{FiltrationTernarySextics}) and Theorem \ref{thm:sepquartics}, Theorem \ref{thm:sepsextics}, we conclude that
\begin{equation} \label{44Final} \Sigma_{4,4}=C_0=C_1=C_2=C_3\subsetneq C_4 \subsetneq C_5 \subsetneq C_6=\mathcal{P}_{4,4}, \mbox{ and}
\end{equation}
\begin{equation} \label{36Final} \Sigma_{3,6}=C_0=C_1=C_2=C_3\subsetneq C_4 \subsetneq C_5 \subsetneq C_6\subsetneq C_7=\mathcal{P}_{3,6}.
\end{equation}
This determines all the strict inclusions in (\ref{ConeFiltration}) for quaternary quartics and ternary sextics.
\section{$(n+1)$-ary Quartics for $n\geq 4$ and $(n+1)$-ary Sextics for $n\geq 3$} \label{quartics}
We now investigate the cone filtration (\ref{ConeFiltration}) in the non-Hilbert cases of $(n+1)$-ary quartics ($n\geq 4$) and $(n+1)$-ary sextics ($n\geq 3$) by generalizing our findings (\ref{44Final}) and (\ref{36Final}), see (\ref{QuarticsFinal}) below. As explained in Section \ref{Structure},  for the proof of Theorem \ref{thm:conesdiffer} and Theorem \ref{thm:conesdifferSextics} below, we first need to prove Theorem \ref{thm:degreeraise}.
		\begin{proof}[\bf Proof of Theorem \ref{thm:degreeraise}] When $\delta=d$, the statement is clear. On the other hand, when $\delta\geq d+1$, then we argue inductively.
		For the purpose of this proof, denote by $C_i^\delta$ and $C_{i+1}^\delta$ the subcones $C_i$ and $C_{i+1}$ respectively of $\mathcal{P}_{n+1,2\delta}$. Likewise, denote by $V_i^\delta$ and $V_{i+1}^{\delta}$ the subvarieties $V_i$ and $V_{i+1}$ respectively of $\PP^{k(n,\delta)}$.
		Choose $f\in C_{i+1}^\delta\backslash C_i^\delta$ and set $g(X):=X_0^2f(X)$ in $\mathcal{F}_{n+1,2(\delta+1)}$. We claim $g\in C_{i+1}^{\delta+1}\backslash C_i ^{\delta+1}$.
		
		Since $f\in C_{i+1}^\delta$, we can fix some $A\in\mathcal{G}^{-1}(f)$ with $q_A(y)\geq 0$ for any $[y]\in V_{i+1}^\delta(\R)$. Set $I:=\{0,\ldots,k(n,\delta)\}$ and expand $A$ to a real matrix $B:=(b_{i,j})_{0\leq i,j\leq k(n,\delta+1)}\in\mathrm{Sym}_{k(n,\delta+1)+1}(\R)$ with $B_I:=(b_{i,j})_{i,j\in I}=A$ and zero elsewhere. Observe that $B\in\mathcal{G}^{-1}(g)$. For any $[z]\in V_{i+1}^{\delta+1}(\R)$ define $[y]:=[z_0:\ldots:z_{k(n,\delta)}]\in V_{i+1}^\delta(\R)$, then $q_B(z)=q_A(y)\geq 0$. This shows $g\in C_{i+1}^{\delta+1}$.
		
		It remains to show $g\not\in C_i^{\delta+1}$. It is enough to verify that $q_B$ is not locally PSD on $V_i^{\delta+1}(\R)$ for any $B\in\mathcal{G}^{-1}(g)$. On contrary, assume that there exists a $B:=(b_{i,j})_{0\leq i,j\leq k(n,\delta+1)}$ such that $q_B(z)\geq 0$ for any $[z]\in V_i^{\delta+1}(\R)$. A comparision of coefficients between $g$ and $q_B(m_0(X),\ldots,m_{k(n,\delta+1)}(X))$ as well as serveral point evaluations of $q_B$ on $V_i^{\delta+1}(\R)$ now reveals that necessarily $b_{i,j}=b_{j,i}=0$ for $j=k(n,\delta)+1,\ldots,k(n,\delta+1)$ and $i=0,\ldots,j$. Hence, the submatrix $A:=B_I:=(b_{i,j})_{i,j\in I}$ of $B$ is a Gram matrix associated to $f$. Since $f\not\in C_i^{\delta}$, it is possible to fix a $[y]\in V_i^\delta(\R)$ with $q_A(y)<0$. Set $[z]:=[y_0:\ldots:y_{k(n,\delta)}:0:\ldots:0]\in V_i^{\delta+1}(\R)$, then $q_B(z)=q_A(y)<0$, which is a contradiction.
		\end{proof}

\begin{proof}[\bf Proof of Theorem \ref{thm:conesdiffer}] For the purpose of this proof, denote by $C_j^3$ the subcone $C_j$ of $\mathcal{P}_{4,4}$ for $j=3,\ldots,6$, and by $C_i^n$ and $C_{i+1}^n$ the subcones $C_i$ and $C_{i+1}$ respectively of $\mathcal{P}_{n+1,4}$. Likewise, set $K_3^3$ to be the affine subvariety $K_3$ of $\C^{k(3,2)}$ and set $K_i^n$ to be the affine subvariety $K_i$ of $\C^{k(n,2)}$. Let $q_{j,3}$ denote the quadratic form $q_j$ in $\mathcal{F}_{k(3,2)}$ for $j\in\{1,2,3,4\}$ and let $q_{j,n}$  denote the quadratic form $q_j$ in $\mathcal{F}_{k(n,2)}$ for $j=1,\ldots,i+1$.
We distinguish three cases:
	\begin{enumerate}
	\item $m_{n+i}(X)=X_jX_n$ for some $j\in\{2,\ldots,n-1\}$
	\item $m_{n+i}(X)=X_jX_l$ for some $j,l\in\{2,\ldots,n-1\}$ such that $j\leq l$
	\item $m_{n+i}(X)=m_{k-n-1}(X)$
\end{enumerate}
The above three cases can be treated by similar arguments for which
it suffices to find $g\in\mathcal{P}_{n+1,4}$ with the following two properties:
\begin{enumerate}
	\item[(a)] There exists a Gram matrix $B$ associated to $g$ such that $q_B$ is locally PSD on $V(q_{1,n},\ldots,q_{i+1,n})(\R)$.
	\item[(b)] For any Gram matrix $B$ associated to $g$, there exists a $[y]\in\phi(K_{i}^n)$ with $q_B(y)<0$.
\end{enumerate}
Let us determine a form $g$ satisfying \textit{(a)} and \textit{(b)} in case \textit{(1)}. 
	\begin{enumerate} 
		\item[(a)]
		Fix $f\in\mathcal{C}_4^3\backslash C_3^3$ and $A\in\mathcal{G}^{-1}(f)$ as in the proof of Theorem \ref{thm:sepquartics}. Set $g(X_0,\ldots,X_n):=f(X_0,X_{j-1},X_j,X_{n})\in\mathcal{P}_{n+1,4}$.  As in the observation given in Section \ref{SecConeFilt}, choose $I\subseteq\{0,\ldots,k(n,2)\}$ appropriately. This allows us to induce the Gram matrix $B:=A_I$ associated to $g$. Define $$\begin{array}{cccc} \pi\colon&\C^{k(n,2)+1} &\to& \C^{k(3,2)+1} \\ &(z_0,\ldots,z_{k(n,2)})&\mapsto& (z_i)_{i\in I}.\end{array}$$ 
		
		For any $[z]\in V(q_{1,n},\ldots,q_{i+1,n})(\R)$, we get $q_B(z)=q_{A}(y)\geq 0$ for $[y]:=[\pi(z)]\in\mathcal{V}(q_{1,3},\ldots,q_{4,3})(\R)$.
		\item[(b)] Let $B\in\mathcal{G}^{-1}(g)$ be arbitrary but fixed and let $A\in\mathcal{G}^{-1}(f)$ be obtained from $B$ using the observation from Section \ref{SecConeFilt}. Since $f\not\in C_3^3$, we can fix some $[y]\in\phi({K}_3^3)(\R)$ with $q_A(y)<0$. Define $[z]\in\phi(K_i^n)(\R)$ with $\pi(z)=y$ and zero elsewhere, then $q_{B}(y)=q_{A}(z)<0$.
		\end{enumerate}
	Case (2) can be verified analogously for $g_1(X_0,\ldots,X_n):=f_1(X_0,X_{1},X_j,X_{l+1})$ with $f_1\in C_5^3\backslash C_4^3$ and case (3) for $g_2(X_0,\ldots,X_n):=f_2(X_0,X_1,X_{n-1},X_{n})$ with $f_2\in C_6^3\backslash C_5^3$.
	\end{proof}
	\begin{proof}[\bf Proof of Theorem \ref{thm:conesdifferSextics}] 
		Theorem \ref{thm:conesdiffer} yields that $C_{n}\subsetneq\ldots\subsetneq C_{k(n,2)-n}$ as a subcone filtration of $\mathcal{P}_{n+1,4}$. Hence, also $C_{n}\subsetneq\ldots\subsetneq C_{k(n,2)-n}$ as a subcone filtration of $\mathcal{P}_{n+1,6}$ by Theorem \ref{thm:degreeraise}. For proving that each inclusion in the remaining subcone filtration $C_{k(n,2)-n}\subsetneq \ldots\subsetneq C_{k(n,3)-n}$ of $\mathcal{P}_{n+1,6}$ is strict,
		we distinguish five cases:
\begin{enumerate}
	\item $m_{n+i}(X)=m_{k(n,2)-n}(X)$
	\item $m_{n+i}(X)=X_j^3$ for some $j,l\in\{1,\ldots,n-2\}$, $j\leq l\leq n-1$
	\item $m_{n+i}(X)=X_jX_lX_n$ for some $j,l\in\{1,\ldots,n-1\}$, $j\leq l$
	\item $m_{n+i}(X)=X_jX_n^2$ for some $j\in\{1,\ldots,n-1\}$
	\item $m_{n+i}(X)=X_jX_lX_r$ for some $j\in\{1,\ldots,n-2\}$, $l,r\in\{j+1,\ldots,n-1\}$, $j+1\leq l$, $l\leq r$
\end{enumerate}
These cases can be treated similarly as the ones in the proof of Theorem \ref{thm:conesdiffer} by reducing to ternary sextics and consideration of extensions $g\in \mathcal{F}_{n+1,6}$ of appropriate separating forms $f\in\mathcal{F}_{3,6}$. These are respectively:
\begin{enumerate}
	\item $g(X_0,\ldots,X_n):=f(X_0,X_1,X_n)$ with $f\in C_4\backslash C_3$
	\item $g(X_0,\ldots,X_n):=f(X_0,X_{j},X_{l+1})$ with $f\in C_5\backslash C_4$
	\item $g(X_0,\ldots,X_n):=f(X_0,X_{j},X_{l+1})$ with $f\in C_6\backslash C_5$
	\item $g(X_0,\ldots,X_n):=f(X_0,X_{j},X_{j+1})$ with $f\in C_7\backslash C_6$
\end{enumerate}
The same argumentation also remains valid for case \textit{(5)}. However, here, we have to reduce to the non-Hilbert case $(4,6)$ and find an appropriate separating form $f\in C_{11}\backslash C_{10}$ with a corresponding extension $g\in\mathcal{F}_{n+1,6}$. For example,
$f(X):=X_1^2\mathfrak{C}^\tau(X)$ and $g(X_0,\ldots,X_n):=f(X_0,X_j,X_l,X_{r+1})$ could be considered.
\end{proof}

For $(n+1,4)_{n\geq 4}$ and $(n+1,6)_{n\geq 3}$, Theorems \ref{thm:collapse}, \ref{thm:conesdiffer} and \ref{thm:conesdifferSextics} together yield
\begin{equation} \label{QuarticsFinal} \Sigma_{n+1,2d}=C_0=\ldots=C_n\subsetneq C_{n+1} \subsetneq \ldots \subsetneq C_{k(n,d)-n}=\mathcal{P}_{n+1,2d}.
\end{equation}

\noindent \textbf{Concluding Remark.} It follows from (\ref{44Final}), (\ref{36Final}) and (\ref{QuarticsFinal}) that the exact number of strictly separating cones in (\ref{ConeFiltration}) is
	\begin{enumerate}
		\item $\mu(n,2):=(k(n,2)-n)-(n+1)$ for $(n+1,4)_{n\geq 3}$, and
		\item $\mu(n,3):=(k(n,3)-n)-(n+1)$ for $(n+1,6)_{n\geq 2}$.
	\end{enumerate}

\section*{Acknowledgements}
We wish to thank \textit{Mathematisches Forschungsinstitut Oberwolfach} for its hospitality. 
The first author acknowledges the support through \textit{Oberwolfach Leibniz Fellows} program. The second author is grateful for the \textit{Oberwolfach Leibniz Graduate Students} award.
The third author acknowledges the support of \textit{Ausschuss für Forschungsfragen} der Universität Konstanz.
The authors are thankful to Maria Infusino and Victor Vinnikov for useful discussions.
\bibliographystyle{alpha}
\bibliography{references2}
\end{document}